\documentclass{sig-alternate}
\usepackage{algorithm}
\usepackage{algorithmic}
\usepackage{comment}
\usepackage{amsmath}   
\usepackage{cite}      
\usepackage{graphicx,psfrag,epsfig,epsf}
\usepackage{colordvi}
\usepackage{subfigure}
\usepackage{algorithm}
\usepackage{algorithmic}
\usepackage{times} 
\usepackage{amsmath} 
\usepackage{amssymb}  
\usepackage{cite}      
\usepackage{psfrag}    
\usepackage{array}
\usepackage{times}
\usepackage{algorithm}
\usepackage{algorithmic}
\usepackage{mathtools}
\usepackage{tikz}

\newtheorem{lemma}{Lemma}
\newtheorem{definition}{Definition}
\newtheorem{theorem}{Theorem}
\newtheorem{proposition}{Proposition}

\newtheorem{assumption}{Assumption}
\newtheorem{remark}{Remark}
\newtheorem{example}{Example}
\newtheorem{rexample}{Running Example Part}

\newcommand{\C}[1]{\mathbb{#1}}

\newcommand{\projection}[1]{\operatorname{lower}({#1})}

\newcommand{\conv}[1]{\operatorname{conv}({#1})}

\newcommand{\permanent}[1]{\operatorname{per}(#1)}

\newcommand{\diagonal}[1]{\operatorname{diag}(#1)}

\newcommand{\cS}{\mathcal{S}}

\newcommand{\cX}{\mathcal{X}}

\newcommand{\cC}{\mathcal{C}}
\newcommand{\cE}{\mathcal{E}}

\newcommand{\cM}{\mathcal{M}}
\newcommand{\cB}{\mathcal{B}}
\newcommand{\cL}{\mathcal{L}}
\newcommand{\cV}{\mathcal{V}}

\newcommand{\cA}{\mathcal{A}}
\newcommand{\cR}{\mathcal{R}}

\newcommand{\cW}{\mathcal{W}}
\newcommand{\cZ}{\mathcal{Z}}
\newcommand{\ok}{\overline{k}}

\newcommand{\cY}{\mathcal{Y}}

\newcommand{\R}{\C{R}}

\newcommand{\N}{\C{N}}

\numberwithin{equation}{section}
\newtheorem{fact}{Fact}

\makeatletter
\def\@copyrightspace{\relax}
\makeatother

\begin{document}

\title{Computing the  domain of attraction  of switching systems
subject to non-convex constraints\titlenote{Research supported by the Belgian Interuniversity Attraction
 Poles, and by the ARC grant 13/18-054 from Communaut\'e francaise de Belgique - Actions de Recherche Concert\'ees.}}

\numberofauthors{2}

\author{
\alignauthor
Nikolaos Athanasopoulos \\
       \affaddr{ICTEAM Institute,}\\
       \affaddr{Universit\'{e} Catholique de Louvain}\\
       \affaddr{4 Avenue Georges Lemaitre,}\\
       \email{nikolaos.athanasopoulos@uclouvain.com}
\alignauthor
Rapha{\"e}l M. Jungers\titlenote{R. M. Jungers is a F.R.S.-FNRS Research Associate.} \\
       \affaddr{ICTEAM Institute,}\\
       \affaddr{Universit\'{e} Catholique de Louvain}\\
       \affaddr{4 Avenue Georges Lemaitre,}\\
       \email{raphael.jungers@uclouvain.com}
}

\maketitle

\begin{abstract} We characterize and compute the maximal admissible positively invariant set for asymptotically stable constrained switching linear systems. Motivated by practical problems found, e.g., in obstacle avoidance, power electronics and nonlinear switching systems, in our setting the constraint set is formed by a finite    number of polynomial inequalities. 
First, we observe that the so-called Veronese lifting allows to represent the constraint set as a polyhedral set.
Next, by exploiting the fact that the lifted system dynamics remains linear, we establish a method based on reachability computations to characterize and compute the maximal admissible invariant set, which coincides with
the domain of attraction when the system is asymptotically stable.
After developing the necessary theoretical background, we propose algorithmic procedures for its exact computation, based on linear or semidefinite programs. The approach is illustrated in several numerical examples.
\end{abstract}

\section{Introduction}
When a set $\cS\subset\R^n$ is invariant\footnote{Throughout the paper and for simplicity, 	we use the terminology `invariant set' for the concept which is usually referred to as  `positively invariant set' \cite{blanchini:1999}.} with respect to a system, all trajectories starting from $\cS$ remain in it forever.
Since almost every system in practice is subject to some type of constraints on its states or outputs, the notion of invariance becomes extremely relevant in control applications \cite{blanchini:1999}.
Specifically, problems related to safety and viability \cite{Aubin:11} can be addressed by computing sets which possess the invariance property or a variant of it.

For linear switching systems, there are at least two approaches one can follow to compute invariant sets, namely  use dynamic programming or find a Lyapunov function and utilise its sub-level sets \cite{Khalil02}.
The mechanism behind the first approach consists in iteratively computing elements of a convergent set sequence generated from the pre-image map, starting from an appropriately chosen initial set, \cite[Ch. 5]{Blanchini08}, \cite{bertsekas:1972,gutman:87,Gilbert91,ABL:14}. The second approach consists in first characterizing non-conservative families of candidate Lyapunov functions and (hopefully) in developing a computational methodology for solving the corresponding conditions. For linear switching systems, polytopic \cite{MOlchPyat:89}, piecewise quadratic \cite{RantzerANA98,HuLin:2003} and sum of squares (sos) polynomial functions \cite{Prajna03,Jungers:09} have been identified as universal, while efficient algorithmic procedures have been established using linear or semidefinite programming \cite{boyd:ghaoui:feron:balakrishnan:1994,Parrilo:00}.

Apart from few exceptions that include the sub-level sets of min-of-quadratics and sos Lyapunov functions, the available constructions concern invariant sets which are convex. This is not restrictive for the stability analysis problem. Moreover, convex shapes recover the maximal invariant set for systems under polytopic constraints such as in Figure~\ref{Ex_figure1}(a), since the convex hull of any invariant set preserves invariance.
\begin{figure}
\centering     

\subfigure[Polytopic set]{\label{fig:a} %
\includegraphics[width=0.23\textwidth]{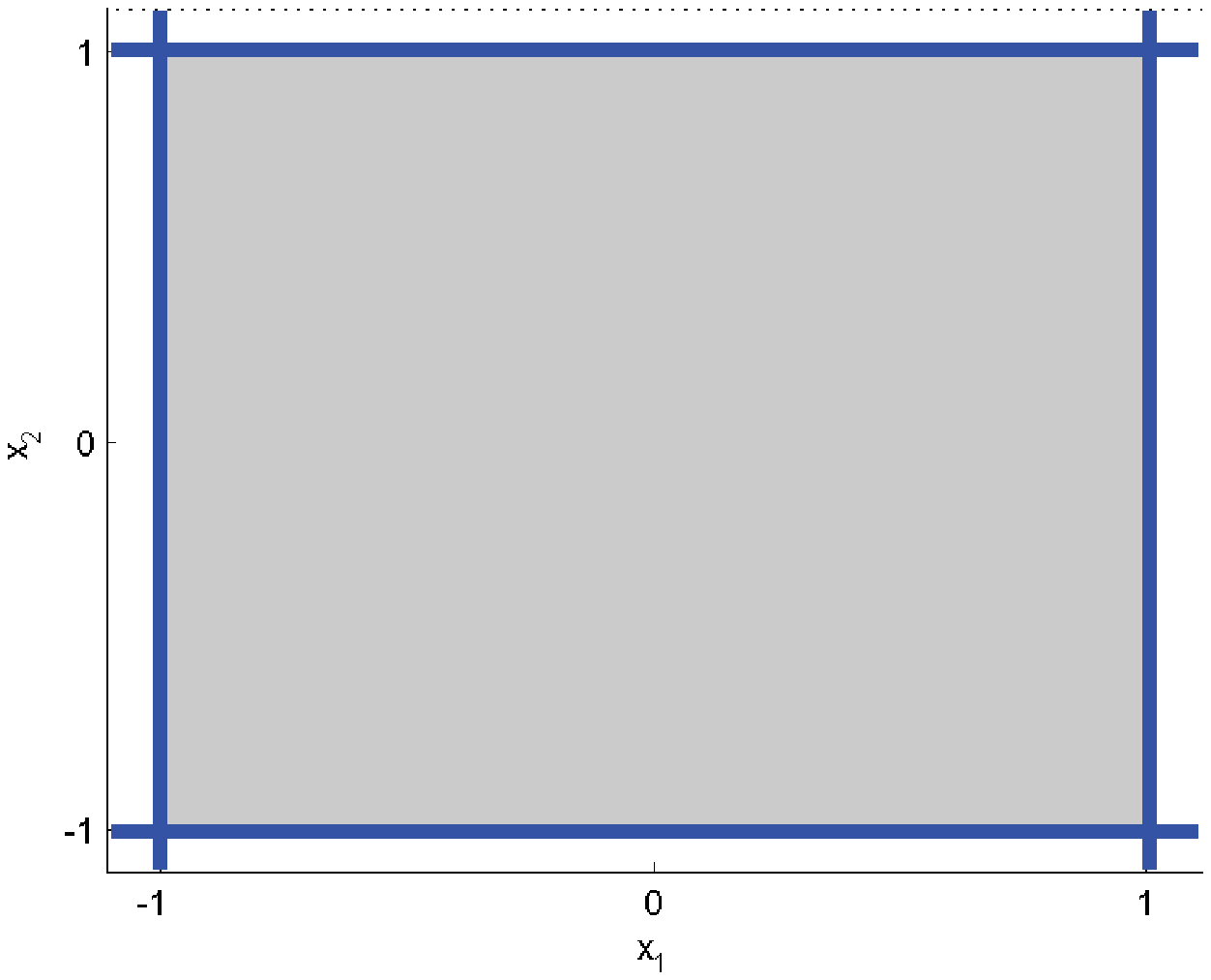}
\psfrag{x2}[][][0.8]{{$\cX$}}
}
\subfigure[Semi-algebraic set]{\label{fig:b}\includegraphics[width=0.23\textwidth]{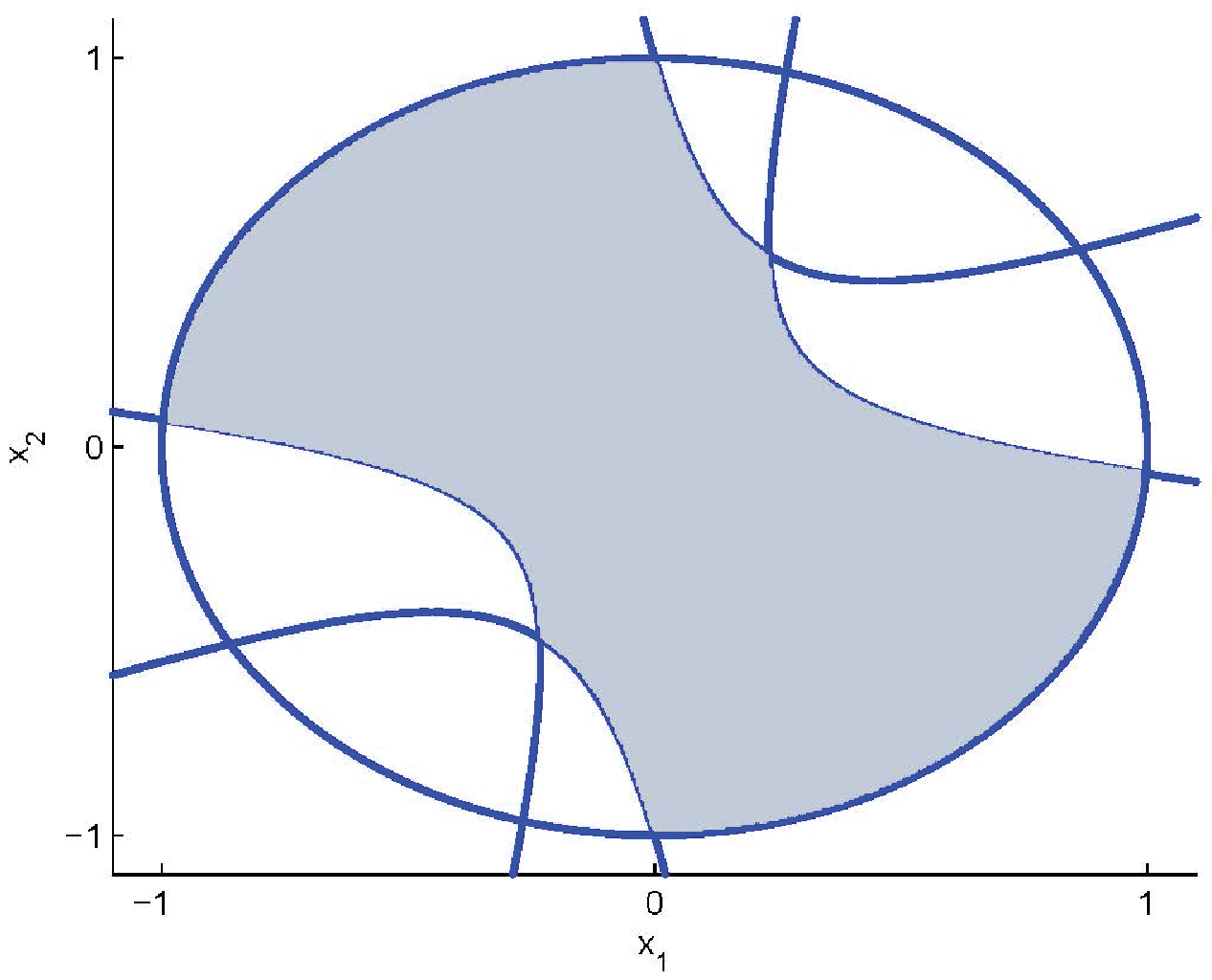}} 
\caption{State constraint sets.}
\label{Ex_figure1}
\end{figure}

Nevertheless, the use of convex invariant sets or Lyapunov functions is restrictive in the setting studied in this paper. Indeed, when the constraint set is semi-algebraic, as for example in Figure~\ref{Ex_figure1}(b), the maximal invariant set does not need to be convex. Furthermore, modifying the standard approaches in order to deal with the non-convex case is not straightforward; it is neither clear  how to  handle non-polytopic sets efficiently in dynamic programming nor how to identify and optimize over families of Lyapunov functions which capture exactly the maximal invariant set.
Additional to the theoretical challenge, the practical motivation for dealing with systems under semi-algebraic constraints  comes from a variety of applications found for example in the path planning and obstacle avoidance  framework \cite{Belta:07}, in power electronics and in non-linear switching systems \cite{AhmadiJungers:13}. 

In this paper we solve both the problems of characterizing the maximal invariant set and of computing it efficiently.
A first helpful observation towards achieving this goal is that semi-algebraic sets are represented by polyhedra in the lifted space induced by the Veronese embedding.
Roughly, the Veronese embedding is a nonlinear mapping of a vector $x\in\R^n$ to a higher dimensional space $\R^N$ defined by the monomials  $x^{\alpha}=[x^{\alpha_1} \quad x^{\alpha_2}  \ldots x^{\alpha_N}]^\top$ that  are of order $d$, where $\alpha_i\in\N^n$ stands for the $n$-tuples that sum up to $d$ and construct each monomial.
This lifting technique has been used with success in the past, see e.g., \cite{zelentsovsky:1994,parrilo:08}, to deal with problems related to stability analysis and approximation of the joint spectral radius of switching systems.


The lifted system enjoys the same stability property with the original system, and more importantly, it remains a switching linear system.
Taking this into account, we are able to establish a relationship between invariant sets in the lifted and original state space.
Additionally, we characterize the maximal invariant set by applying a variant of the \emph{backward reachability algorithm} \cite{Aubin:11,Blanchini08}  in the lifted space.
The corresponding set sequence may be initialized either with the lifted constraint set or with the, possibly unbounded, polyhedral set that is induced from the semi-algebraic constraint set.
We address two specific challenges that arise depending on each choice, namely how to efficiently compute the reachability mapping in the former case and how to guarantee convergence in the latter case.
We show that the maximal admissible invariant set is well-defined, it can be computed in a finite number of steps and it is expressed as the unit sub-level set of a max-polynomial function consisting of a finite number of pieces. To this end, we establish three possible algorithmic implementations for computing the maximal invariant set  based on linear or semidefinite programs. To the best of our knowledge, this is the first time that the exact computation of the domain of attraction under non-convex constraints is possible.

Finally, it is worth to distinguish between the different research objectives set in this work from the ones found in the sos framework, see for example \cite{Papa:05}, where more complex dynamics and constraints are studied. The problem studied there concerns the assessment of local asymptotic stability in the neighborhood of the equilibrium point, however, no guarantee on the level of the approximation of the domain of attraction is sought or provided. Another distinction  should be made with the work in \cite{AhmadiJungers:13}, where the focus is restricted to computing convex invariant approximations of the domain of attraction.

In section~\ref{section2}, the basic definitions and the problem setting are presented, together with the technical details regarding the procedure of lifting the system and the constraint set.
In section~\ref{section3}, we characterize the maximal admissible invariant set by first associating the invariance properties of sets in the lifted and original space and next by applying a modified version of the backward reachability algorithm. The corresponding algorithms are presented in section~\ref{section4}. In section~$5$ two numerical examples are presented, whereas conclusions are drawn in section~$6$. Finally, further details concerning the algorithmic implementation of the results are exposed in the Appendix.

\section{Preliminaries}\label{section2}

\subsection{Notation}\label{subsection2.1}

We denote the field of real
numbers and the set of non-negative integers with $\R $ and $\N$ respectively.
We write vectors $x,y$  with small letters and sets $\cS,\cX,\cV$  with capital letters in italics.
The vector in $\R^n$ with all elements equal to one is denoted by $1_n$. For matrices and vectors, inequalities hold component-wise.
Given a $n$-tuple $\alpha\in\N^{n}$, the $\alpha$ monomial of a vector $x\in\R^n$ is  $x^{\alpha}=x_1^{\alpha_1}\ldots x_n^{\alpha_n}$.
The degree of the monomial is $d=\sum_{i=1}^n\alpha_i$. We denote by $\alpha !$ the multinomial coefficient $\alpha !=\frac{d!}{\alpha_1 ! \ldots \alpha_n !}$.

\subsection{Setting and problem formulation} \label{subsection2.2}
Let $\cA:=\{ A_1,...,A_M  \}\subset\R^{n\times n}$ be a set consisting of $M$ matrices.
The system under study is
\begin{equation}\label{eq_sys}
x(t+1)=A_{\sigma(t)}x(t),
\end{equation}
where $x(0)\in\R^n$, $t\in\N$ and the switching signal $\sigma(\cdot):\N\rightarrow \{1,...,M\}$ assigns at each time instant a matrix from the set $\cA$.
The System \eqref{eq_sys} is subject to state constraints
\begin{equation}\label{eq_constraints}
x(t)\in\cX, \quad t\geq 0.
\end{equation}
 The state constraint set is of the form
\begin{equation}\label{eq_set}
\cX:=\{x\in\R^n: c_i(x)\leq 1, i=1,...,p \},
\end{equation}
where $c_i(\cdot):\R^n\rightarrow \R$, $i=1,..,p$, are polynomials of maximum degree $d\geq 1$.
We are interested in characterizing the domain of attraction for the linear switching System \eqref{eq_sys} subject to constraints \eqref{eq_constraints}.
Throughout the paper, we make the following assumptions.
\begin{assumption}\label{assumption1}
The System \eqref{eq_sys} is asymptotically stable.
\end{assumption}
\begin{assumption}\label{assumption2}
The set $\cX\subset\R^n$ \eqref{eq_set} is closed, bounded and contains the origin in its interior.
\end{assumption}
Assumption~\ref{assumption1} does not affect the generality of the problem since the admissible domain of attraction is different from the singleton set $\{ 0\}$ only if the switching linear System \eqref{eq_sys} is asymptotically stable.
Moreover, under Assumptions~\ref{assumption1} and \ref{assumption2}, the admissible domain of attraction coincides with the maximal admissible invariant set.
The assumption that the origin is in the interior of the constraint set $\cX$ in Assumption~\ref{assumption2} is a technical one, and it is required in the proofs of Theorems~\ref{theorem1}-\ref{theorem3}. It is worth mentioning that this assumption is taken in the standard problem of computing the maximal admissible invariant set for linear switching systems under polytopic constraints \cite{Blanchini08}, while its removal, even when the constraint set is a polyhedron is still being investigated, see e.g., \cite{BitsSorin:13}.

\begin{definition}
A set $\cS\subset\R^n$ is called \emph{invariant} with respect to the System \eqref{eq_sys} if $x(0)\in\cS$ implies $x(t)\in\cS$ for all $t\in\N$ and any switching signal $\sigma(\cdot):\N\rightarrow \{1,...,M\}$.
Moreover, if $\cS\subseteq\cX$, the set $\cS$ is called an \emph{admissible invariant set}   with respect to the System \eqref{eq_sys} and the constraints \eqref{eq_constraints}.
\end{definition}
\begin{definition}\label{def_maximal}
The set $\cM\subset\R^n$ is called the \emph{maximal admissible invariant set} with respect to the System \eqref{eq_sys} and the constraints \eqref{eq_constraints}
if it is admissible invariant, and, moreover, for any admissible invariant set $\cS\subseteq\cX$, it holds that $\cS\subseteq \cM$ .
\end{definition}
Thus, the problem investigated in this paper is naturally formulated as follows:
Suppose that Assumptions~\ref{assumption1} and \ref{assumption2} hold.
Compute the maximal admissible invariant set with respect to the System \eqref{eq_sys} and the state constraints \eqref{eq_constraints}.

\subsection{Lifting the system}
We now describe formally the algebraic lifting applied to System \eqref{eq_sys},
resulting in a dynamical system which enjoys the same stability properties.
The broad idea is to construct monomials of $x$ of a certain maximum degree $d$ and infer properties of our dynamical system from the one obtained after  this state-space transformation.

\begin{definition} {\emph{\cite{parrilo:08}, \cite{Jungers:09}}. }
Given a vector $x\in\R^n$ and an integer $d\geq 1$, the $d$-lift of $x$, denoted by $x^{[d]}$, is the vector in $\R^{n+d-1\choose d}$,  having as elements all the exponents $\alpha$ of degree $d$, i.e,.
\begin{equation*}
x_{\alpha}=\sqrt{\alpha !}x^{\alpha}.
\end{equation*}
\end{definition}
\begin{definition} {\emph{\cite{parrilo:08}, \cite{Jungers:09}}. }
Given $\cA\subset\R^{n\times n}$ and an integer $d\geq 1$, the $d$-lift of the set $\cA$ is $\cA^{[d]}:=\{A_1^{[d]},\ldots A_{M}^{[d]}\}\subset\R^{{n+d-1\choose d}\times{n+d-1\choose d}}$ where each matrix $A_i^{[d]}$, $i=1,...,M$, is associated to the linear map\footnote{ One can obtain a numerical expression of the entries of $A^{[d]}$ with the formula
$A_{\alpha\beta}^{[d]}=\frac{\permanent{A(\alpha,\beta)}}{\sqrt{\mu(\alpha)\mu(\beta)}},$
where  $\mu(\alpha)$ is the product of the factorials of the entries of $\alpha$, the matrix $\overline{A}=A(\alpha,\beta)\in\R^{n\times n}$ has elements  $\overline{a}_{ij}:=a_{\alpha_{i}\beta_{j}}$,
$i\in [1,n]$, $j\in [1,m]$ and
 $\permanent{A}=\sum_{\pi\in S_n}\cdot$ $\prod\limits_{i=1}^n a_{i,\pi(i)}$ is the permanent of a matrix $A\in\R^{n\times n}$, where $S_n$
 is the symmetric group on $n$ elements.} $A^{[d]}_i:x^{[d]}\rightarrow (A_ix)^{[d]}$.
\end{definition}
In what follows, we define a natural extension of the $d$-lift  which is generated by
 stacking the $l$-lifts of a vector, for a set of integers $l$, in a single augmented vector.
To this end, let us consider the ordered set of integers $\cL:=\{ l_1,l_2,...,l_K\}$, $l_i\in[1,d]$, $i\in[1,K]$, where $K\leq d$.
\begin{definition} \label{definition3}
Given an integer $d\geq 1$, the set $\cL:=\{l_1,...,l_K\}$, $ K\leq d$  and a vector $x\in\R^n$,
the $\cL$-lift of $x$, denoted by $x^{[\cL]}\in\R^N$, $N=\sum_{l_i\in\cL} {n+l_i-1 \choose l_1}$ is
\begin{equation*}
x^{[\cL]}:=\left[x^{[l_1]\top} \quad  x^{[l_2]\top} \ldots x^{[l_K]\top}  \right]^\top.
\end{equation*}
Similarly, the $\cL$-lift of the set $\cA$ is $\cA^{[\cL]}:=\{A_{1}^{[\cL]},\cdots, A_{M}^{[\cL]} \}\subset\R^{N\times N}$, where
\begin{equation*}
A_{i}^{[\cL]}:=\diagonal{A_i^{[l_1]},\ldots, A_{i}^{[l_K]}}, \quad i=1,...,M.
\end{equation*}
\end{definition}
We define the $\cL$-lifted system
\begin{equation}\label{eq_sysd!}
y(t+1)=A_{\sigma(t)}^{[\cL]}y(t),
\end{equation}
where $y_0\in\R^{N}$, $N=\sum_{i=1}^K{n+l_i-1 \choose l_i}$, $t\in\N$ and $\sigma(\cdot):\N\rightarrow \{1,...,M\}$  is the switching signal.
System \eqref{eq_sysd!} can simply be considered to be generated by stacking the $[l_i]$-lifts of \eqref{eq_sys} for all $i\in [1,K]$.
The properties below follow from the definition of a $d$-lift.
\begin{fact} \label{fact1}
Consider an integer $d\geq 1$, the ordered set of integers $\cL=\{l_1,...,l_K\}$, $l_i\in [1,d]$, $K\leq d$ and a matrix $A\in\R^{n\times n}$. Then, for any $x\in\R^n$, it holds that
\begin{align*}
(Ax)^{[d]} & = A^{[d]}x^{[d]}, \\
(Ax)^{[\cL]} & = A^{[\cL]}x^{[\cL]}. 
\end{align*}
\end{fact}
We make use of the following notion, which formalizes the stability notion for a linear switching system.
\begin{definition}\emph{\cite{RotaStrang:60}, \cite{Jungers:09}.}
The joint spectral radius of a matrix set $\cA\subset\R^{n\times n}$ is equal to
\begin{equation}\label{eq_jsr}
\rho(\cA):=\lim_{t\rightarrow \infty}\max\{\| A\|^{\frac{1}{t}}: A\in\cA^{t}  \}.
\end{equation}
\end{definition}
The switching System \eqref{eq_sys} is asymptotically stable if and only if $\rho(\cA)<1$ \cite{Jungers:09}.
\begin{proposition}\label{proposition1}
The System \eqref{eq_sys} is globally absolutely exponentially stable (GAES) if and only if the System \eqref{eq_sysd!} is globally absolutely exponentially stable.
\end{proposition}
\begin{proof}
For any $j\geq 1$, it holds that $\rho(\cA)^j=\rho(\cA^{[j]})$  \cite{BloNes:05}. Moreover, since the matrices $\cA^{[l_i]}$ are block diagonal, $i\in[1,K]$, it holds \cite{Jungers:09}
\begin{equation*}
\rho(\cA^{[\cL]})=\max_{i\in [1,K]} \{ \rho(\cA^{[l_i]})   \}=\max_{i\in [1,K]}\{\rho^{l_i}(\cA)\}.
\end{equation*}
Consequently,  $\rho(\cA)< 1$ if and only if $\rho(\cA^{[\cL]})< 1$.
We finish the proof by recalling the equivalence between asymptotic and exponential stability for homogeneous systems, see e.g., \cite[Corollary V.3]{MDA:13}, of which switching linear systems are a subclass,
and that the switching System \eqref{eq_sys} is  GAES if and only if $\rho(\cA)<1$~\cite{Jungers:09}.
\end{proof}
\begin{rexample}
 Let us consider a two-dimensional system  \eqref{eq_sys} consisting of two modes, i.e., $\cA:=\{A_1,A_2\}$,
with $A_1=\begin{bmatrix}  1.0425 & 0.3416 \\ -0.5893  &  0.5839\end{bmatrix}$, $A_2=\begin{bmatrix}  0 & 0.6500 \\ 0.6500  &  0\end{bmatrix}$.
Let $\cL=\{2\}$. Following Definition~\ref{definition3}, the $\cL$-lift of $x$ is
$$x^{[\cL]}=x^{[2]}=[x_2^2\quad \sqrt{2}x_2x_1\quad x_1^2]^\top,$$
while $\cA^{[\cL]}=\{A_1^{[2]},A_2^{[2]}\}$, with (rounded up to the second digit)
$$A_1^{[2]}=\begin{bmatrix}  0.34 & -0.49 & 0.35 \\ 0.28  &  0.40 & -0.87 \\
0.12 & 0.50 & 1.09\end{bmatrix}, A_2^{[2]}=\begin{bmatrix}  0 & 0 & 0.42 \\ 0 & 0.42 & 0 \\ 0.42 & 0 & 0  \end{bmatrix}.$$
Using the JSR Toolbox \cite{Raphael:14}, we calculate the joint spectral radius of the matrix set $\cA$ to be to $0.9$ with accuracy $9\cdot 10^{-8}$, thus the system \eqref{eq_sys} is asymptotically stable.
As expected from Proposition~\ref{proposition1}, the joint spectral radius of the set $\cA^{[2]}$ is found equal to $0.81$ with accuracy $7.64\cdot 10^{-7}$, thus the system \eqref{eq_sysd!} is also asymptotically stable.
\end{rexample}

\subsection{Lifting the constraints}
We consider the set $\cX$ \eqref{eq_set} and denote with
 $\cL_i\subseteq [1,d]^{K_i} $, $i\in [1,p]$, $K_i\leq d$ the index sets that correspond to the degrees of all monomials appearing in each function $c_i(x)$.
Also, we let $\cL\subseteq [1,d]^d$ contain all the elements of the index sets $\cL_i$, $i=1,...,p$.
We can write each polynomial function $c_i(x)$, $i\in[1,p]$, as a sum of positively homogeneous polynomials $c_{i,l}(x)$ of degree $l\in\cL$, i.e.,
\begin{equation*}
c_{i}(x)=\sum_{l\in\cL_i}c_{i,l}(x).
\end{equation*}
In addition, we
can express  each homogeneous polynomial $c_{i,l}(x)$, $i\in [1,p]$, $l\in\cL_i$, as a linear function of the $\cL$-lifted vectors $x^{[j]}$, $j\in\cL$ as follows
\begin{equation} \label{eq_seteq}
c_i(x):=\sum_{l \in \cL_i  }^dg_{i,l}^\top x^{[l]}=g_{i}^\top x^{[\cL]}, \quad i\in[1,p],
\end{equation}
where $g_{i,l}^\top x^{[l]}:=c_{i,l}(x)$, $l\in\cL_i$. Also,
we have that $g_{i}:=\left[g_{i,l_1}^\top \quad \ldots \quad g_{i,l_K}^\top	 \right]^\top, $ $g_i\in\R^{N}$, $i=1,..,p$, where
\begin{equation}\label{eq_N}
N:=\sum_{l\in\cL}{n+l-1\choose l}.
\end{equation}
We are in a position to define the $\cL$-lift of a set $\cS\subset\R^n$.
\begin{definition}
Consider the set $\cX\subset\R^n$ \eqref{eq_set} that satisfies Assumption~\ref{assumption2}. Let $\cL\subset [1,d]^K$, be the ordered set of integers containing the degrees of  all monomials appearing in $c_i(x)$, $i\in[1,p]$, and $g_{i,l}$, $i\in[1,p]$, $l\in\cL$ be vectors satisfying \eqref{eq_seteq}. We define the $\cL$-lift of the set $\cX$ as $\cX^{[\cL]}\subset\R^N$, where $N$ is given in \eqref{eq_N}, as
\begin{equation} \label{eq_set_lifted}
\cX^{[\cL]}:=\left\{  y \in\R^N:  g_i^{\top}y\leq 1, i=1,...,p \right\}.
\end{equation}
\end{definition}

Moreover, we define the manifold $\cV\subset\R^N$ which is an algebraic variety,
\begin{equation}\label{eq_manifold}
\cV:=\left\{y\in\R^{N}: \left(\exists x\in\R^n: y=x^{[\cL]}\right) \right\}.
\end{equation}
Taking into account Fact~\ref{fact1}, we can show that the set $\cV$ \eqref{eq_manifold} is invariant with respect to the lifted System~\eqref{eq_sysd!}.

\begin{rexample}
Let us consider as constraint set $\cX$ \eqref{eq_set} the set depicted in Figure~\ref{Ex_figure1}(b).
For this case, the polynomials $c_i(x)$, $i=1,2,3$ that define the set are
\begin{align*}
c_1(x) & = x_1^2+x_2^2, \\
c_2(x) & = x_2^2+6\sqrt{2}x_1x_2-4x_1^2, \\
c_3(x) & = -3x_2^2+10\sqrt{2}x_1x_2+2x_1^2.
\end{align*}
We have $\cL=\cL_1=\cL_2=\cL_3=\{2\}$, and consequently, $\cX^{[\cL]}\in\R^3$ is
given by \eqref{eq_set_lifted}, with $g_1=[1\quad 0 \quad 1]^\top$,
$g_2=[1 \quad 6 \quad -4]^\top$, $g_3=[-3 \quad 10\quad 2]^\top$.
 The set $\cX^{[2]}$ is an unbounded polyhedron and its defining hyperplanes are depicted in Figure~\ref{Ex_figure2} in red.
The set $\cV\cap\cX^{[2]}$ is also shown in Figure~\ref{Ex_figure2} in grey.
\begin{figure}[h]
\begin{center}
\includegraphics[height=4.2cm]{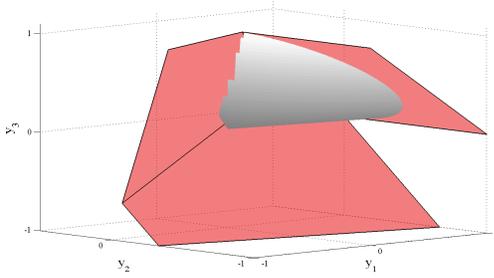}
\end{center}
\caption{The lifted semi-algebraic set $\cX^{[2]}\cap\cV$ of Figure~\ref{Ex_figure1}(b) is depicted in grey. The constraint set is bounded tightly by the polyhedron $\cX^{[2]}$, defined
by the intersection of three half-spaces and shown in red.}
\label{Ex_figure2}
\end{figure}
\end{rexample}

\section{Characterization of the maximal admissible invariant set} \label{section3}

The set $\cX$ \eqref{eq_set} is invariant with respect to the System \eqref{eq_sys} if and only if
\begin{align*}
&  \forall x\in\R^n,\forall j\in[1,M], \\
&(\forall i\in [1,p], c_i(x)\leq 1)  \Rightarrow (\forall i\in[1,p], c_i(A_jx)\leq 1).
\end{align*}
If $c_i(\cdot)$, $i\in [1,p]$, are linear functions, it is well known that invariance can be verified by solving a linear program \cite{Bitsoris:88}.
If the functions $c_i(x)$, $i\in[1,p]$, are positive definite quadratic functions, then invariance can be verified by solving a convex quadratic program \cite{LinAnts:09}.
In comparison, in this paper we aim to find a way to verify and compute invariant sets when the functions $c_i(\cdot)$, $i\in[1,p]$, are general polynomial functions.

In what follows we show that the projection of an admissible invariant set $\cS\subseteq\cX^{[\cL]}$ of the $\cL$-lifted system on $\R^n$ is invariant for the system under study.
To this end, we define the ``reverse'' operation of lifting.
\begin{definition}
Given an index set $\cL\subseteq [1,d]^d$,  and a set in the $\cL$-lifted space $\cS\subseteq \R^{N}$, $N=\sum_{l\in\cL}{n+l-1\choose l}$, the lowering operation of $\cS$ to $\R^n$  is
\begin{equation*}
\projection{\cS}:=\left\{x\in\R^n: \left(\exists y\in\cS: y=x^{[\cL]}\right) \right\}.
\end{equation*}
\end{definition}
Taking into account \eqref{eq_manifold}, it is not difficult to see that the relation
$$\projection{\cS}=\projection{\cS\cap\cV}$$
holds for any set $\cS\subset\R^{N}$. 

\begin{proposition}\label{lemma1}
Consider the System \eqref{eq_sys} and the constraint set \eqref{eq_set}. If $\cS\subseteq\cX^{[\cL]}\subset\R^{N}$,
\begin{equation} \label{eq_lem1_1}
\cS:=\{ y\in\R^{N}: f_{i}^\top y\leq 1, i=1,...,q  \},
\end{equation}
where  $N=\sum_{l\in\cL}{n+l-1\choose l}$, $f_i\in\R^N$, $i\in[1,q]$,
is an admissible invariant set with respect to System \eqref{eq_sysd!} and the constraint set $\cX^{[\cL]}$ \eqref{eq_set_lifted}
then the set $\projection{\cS}$
 is an admissible invariant set with respect to System \eqref{eq_sys} and the constraint set~\eqref{eq_set}.
\end{proposition}
\begin{proof}
Since $\cS\subseteq\cX^{[\cL]}$, it follows that $\cS\cap\cV\subseteq\cX^{[\cL]}\cap\cV$, and consequently,
 $\projection{\cS\cap\cV}\subseteq \projection{\cX^{[\cL]}\cap\cV}=\cX$.
Next, we show that $\projection{\cS}$ is invariant. By construction, $\projection{\cS}=\{x\in\R^n: b_i(x)\leq 1 \}$, where $b_i(x):=f^\top_i x^{[\cL]}$, $i\in[1,q]$.
From hypothesis, for all $y\in\cS$, relation $f_i^{\top}y\leq 1$, $i\in[1,q]$, implies $f_i^\top A^{[\cL]}_jy\leq 1$, for all $j\in[1,M]$.
By definition, for any $x\in\projection{\cS}\subseteq\cX$, there exists a vector $y\in\cS\cap\cV$ such that $y:=x^{[\cL]}$.
Thus, we have
\begin{align*}
f^{\top}_iA^{[\cL]}_jy & =f^{\top}_iA^{[\cL]}_jx^{[\cL]}=f^{\top}_i(A_jx)^{[\cL]}=b_i(A_jx)\leq 1.
\end{align*}
Consequently, $b_i(x)\leq 1$ for all $i\in [1,p]$ implies $b_i(A_jx)\leq 1$ for all $i\in [1,q]$, for all $j\in[1,M]$, and the set $\projection{\cS}$ is admissible invariant with respect to the System \eqref{eq_sys}.
\end{proof}

\begin{remark}\label{rem_nec_and_suf}
It is worth underlining that the statement of Proposition~\ref{lemma1} becomes both necessary and sufficient when $\cS\subseteq\cX^{[\cL]}$ is any set lying on  $\cV$, i.e., when $\cS\cap\cV=\cS$.
\end{remark}

\begin{remark} \label{remark1.5}
The lowering operation is straightforward when $\cS$ is a polyhedron \eqref{eq_lem1_1}, since in this case $\projection{\cS}=\{x\in\R^n: c_i(x)\leq 1, i=1,...,q  \}$, where $c_i(x)=f_i^\top x^{[\cL]}$, $i\in[1,q]$.
\end{remark}

Proposition~\ref{lemma1} suggests that in order to compute invariant sets for the original system and constraint set \eqref{eq_set},
one can first compute admissible invariant sets with respect to the $\cL$-lifted System \eqref{eq_sysd!} and the $\cL$-lifted constraint set \eqref{eq_set_lifted} and consequently perform a projection on the original space.
This observation provides a potential advantage. Indeed, since the System \eqref{eq_sysd!} is a switching linear system and $\cX^{[\cL]}$ \eqref{eq_set_lifted} is a polyhedral set,
one can apply established results for checking invariance of a given polyhedral set.

\begin{proposition}\label{lemma2}
Consider the System \eqref{eq_sys} and the set $\cX$ defined in \eqref{eq_set}. Let $G\in\R^{p\times N}$ be the matrix having as rows the vectors $g_{i}^\top$, $i\in[1,p]$ that describe the set $\cX^{[\cL]}$, defined in  \eqref{eq_set_lifted}. 
Then, the set $\cX$ is invariant with respect to \eqref{eq_sys} if there exist non negative matrices $H_i\in\R^{p\times p}$, $i\in[1,M]$, that satisfy the relations
\begin{align}
GA_{i}^{[\cL]} & =H_iG, \quad i\in[1,M]		\label{eq_lem2_1}, \\
H_i1_p & \leq 1_p \quad   i\in [1,M] \label{eq_lem2_2},\\
H_{i} & \geq 0, \quad i\in [1,M] \label{eq_lem2_3}.
\end{align}
\end{proposition}
\begin{proof}
Conditions \eqref{eq_lem2_1}-\eqref{eq_lem2_3}  are necessary and sufficient for the set $\cX^{[\cL]}$ to be invariant with respect to the System \eqref{eq_sysd!}
\cite{Bitsoris:88}, \cite{Hennet:95}. 
Consequently, from Proposition~\ref{lemma1}, the set $\cX=\projection{\cX^{[\cL]}}$ is invariant with respect to the System \eqref{eq_sys}.
\end{proof}
The algebraic relations \eqref{eq_lem2_1}-\eqref{eq_lem2_3}  can be solved by linear programming. However, although these conditions are necessary and sufficient for  a polyhedral set $\cX^{[\cL]}$ to be invariant with respect to the lifted System \eqref{eq_sysd!}, they are only sufficient for $\cX$ to be invariant w.r.t. the original System \eqref{eq_sys}.
Additionally, since it is impossible to define a polyhedron $\cX^{[\cL]}$ lying on the manifold $\cV$, we cannot exploit Remark~\ref{rem_nec_and_suf} to
pose necessary and sufficient conditions of invariance for $\cX$ w.r.t. \eqref{eq_sys} via $\cX^{[\cL]}$.
Also, apart from the above observations, it might happen that the set $\cX$ is not invariant and consequently the maximal admissible invariant set is a subset of $\cX$. Thus,
exploiting Proposition~\ref{lemma2} to characterize an invariant set is limited.

\begin{rexample}
Let us consider the lifted system and the set $\cX^{[\cL]}$ calculated in the previous parts of the Running Example. In order to verify if $\cX^{[\cL]}$ is an invariant set we utilise Proposition~\ref{lemma2}. To this end, by setting $$G=\begin{bmatrix}
1 & 0 & 1 \\ 1 & 6 & -4 \\ -3 & 10 & 2 \end{bmatrix},$$ constructed from the vectors $g_{i}$, $i=1,2,3,$ that define the set $\cX^{[\cL]}$, we solve the  optimization problem
$$\min_{\varepsilon,H_1,H_2}{\varepsilon}$$
subject to  \eqref{eq_lem2_1},\eqref{eq_lem2_3} and inequalities $H_1 1_3\leq  \varepsilon 1_3$, $H_21_3\leq \varepsilon 1_3$, The optimization problem is infeasible, thus, the set $\cX^{[\cL]}$ is not invariant with respect to \eqref{eq_sysd!}, and consequently, we cannot decide if $\cX$ is invariant with respect to \eqref{eq_sys}.
\end{rexample}

For linear switching systems under polytopic constraints, one can apply well known iterative reachability-based procedures to construct the maximal invariant set, see, e.g., \cite{Blanchini08}.
The approach taken in this paper follows a similar path. In specific, in order to recover the maximal admissible invariant set,we would like to characterize the fixed point of a set sequence generated by applying the pre-image map of the $\cL$-lifted System \eqref{eq_sysd!} for two different initial condition, namely the $\cL$-lifted set $\cX^{[\cL]}$ \eqref{eq_set_lifted} or $\cX^{[\cL]}\cap\cV$. Nevertheless, two issues not present in the standard reachability analysis approach have to be taken into account: On the one hand, as illustrated in the Running Example and Figure~\ref{Ex_figure2}, the set $\cX^{[\cL]}$ might be unbounded, thus, convergence to the maximal invariant set cannot be guaranteed when starting from the set $\cX^{[\cL]}$. On the other hand, when starting from the set $\cX^{[\cL]}\cap\cV$, one has to account for computations of the  reachability operations involving non-polytopic sets.
We address these two challenges in the remaining of the paper.

\begin{definition}
The pre-image map of a set $\cS\subset\R^{N}$, $N=\sum_{l\in\cL}{n+l-1\choose l}$ with respect to System \eqref{eq_sysd!}  is
\begin{equation}\label{eq_preimage}
\cC(\cS):=\left\{ y\in\R^N: A_i^{[\cL]} y \in\cS, \forall i\in[1,M]\right\}.
\end{equation}
\end{definition}
Next, let us consider the set sequence $\{\cS_i\}_{i\geq 0}$ generated by the iteration
\begin{align}
\cS_0 & \subset\R^N, \label{eq_setseq1}\\
\cS_{i+1}& :=\cC(\cS_i)\cap\cS_0, \label{eq_setseq2}
\end{align}
where $N=\sum_{l\in\cL}{n+l-1\choose l}$ and $\cX^{[\cL]}$ denotes the $\cL$-lift of the set \eqref{eq_set}.
In what follows, we will show convergence of the set sequence to the maximal invariant set choosing different initial condition $\cS_0$ \eqref{eq_setseq1}.

\begin{fact} \label{fact2}
Let $\cX\subset\R^n$ \eqref{eq_set} be a semi-algebraic set satisfying Assumption~\ref{assumption2}.
Then, the set $\cV\cap\cX^{[\cL]}$ is compact.
\end{fact}
\begin{proof}
Since $\cV\cap{\cX^{[\cL]}}=\{y\in\R^N: (\exists x\in\cX: y=x^{[\cL]})\}$, the statement follows because the continuous polynomial map of a compact set is compact.
\end{proof}

\begin{theorem}\label{theorem1}
Consider the System \eqref{eq_sys}, the constraint set \eqref{eq_set} and the set sequence $\{\cS_i\}_{i\geq 0}$ generated by \eqref{eq_setseq2} with $$\cS_0:=\cV\cap\cX^{[\cL]}.$$
Then, there exists a finite integer $\ok\geq 1$  such that $$\cS_{\ok}=\cS_{\ok+1}$$ and  the maximal admissible invariant set  $\cM$ with respect to the System \eqref{eq_sys} and the constraints \eqref{eq_set} is $\cM=\projection{\cS_{\ok}}.$
\end{theorem}
\begin{proof}
 Under Assumption~\ref{assumption1} and from Proposition~\ref{proposition1}, there exist  scalars $\Gamma\geq 1$, $\varepsilon\in (0,1)$
such that $\| y(t)\|\leq \Gamma \varepsilon^t\| y(0)\|$, for all $y(0)\in\R^{N}$, for all $y(t)$ satisfying \eqref{eq_sysd!} and for all $t\geq 0$.
From Fact~\ref{fact2}, there exists a number $R>0$ such that $\|y(0)\|\leq R$, for all $y\in\cS_0$.
Consider the set $\cR=\{y\in\R^N: \|y\|\leq R \}$, the  number $a\in\R$, where
$$a:=\max\{\lambda: \lambda\cR\cap\cV\subseteq \cS_0 \},$$
 and the integer $\ok=\left\lceil \log_{\varepsilon} \frac{a}{\Gamma} \right\rceil.$
Then, $y(0)\in\cS_0$ implies $y(t)\in\cS_0$, for all $t\geq \ok$.
On the other hand, for any $t\geq 0$, the relation $y(t)\in\cS_0$ holds for all $y(0)\in\cS_0$ for which  $y(0)\in\cS_t$.
Let us assume that there exists a vector $y(0)\in\cS_{\ok}$ such that $y(0)\notin\cS_{\ok+1}$. This implies that $y(\ok+1)\notin \cS_0$
which is a  contradiction, thus, $\cS_{\ok+1}\supseteq \cS_{\ok}$.
From \eqref{eq_setseq2}, it holds that $\cS_1\subseteq \cS_0$. Suppose that $\cS_{i+1}\subseteq\cS_{i}$. Then, we have that $\cC(\cS_{i+1})\cap\cS_0\subseteq\cC(\cS_{i})\cap\cS_0$, or $\cS_{i+2}\subseteq\cS_{i+1}$. Consequently, $\cS_{\ok+1}\subseteq \cS_{\ok}$, thus, $\cS_{\ok}=\cS_{\ok+1}$.

Next, we show that $\cM$ is the maximal invariant set. By construction it holds that $\cS_{\ok}\subseteq\cS_0$, thus, $\cM=\projection{\cS_{\ok}}\subseteq\projection{\cS_0}=\cX$.
Moreover, for any $x_0\in\cM$, there exists a $y_0\in\cS_{\ok}$ such that $y_0=x_0^{[\cL]}$. Since $\cS_{\ok}=\cS_{\ok+1}$, it holds that $A_i^{[\cL]}y_0\in\cS_{\ok}$, for all $i\in [1,M]$ or, $(A_{i}x_0)^{[\cL]}\in\cS_{\ok}$, which implies $A_i x_0\in\cM$, for all $i\in[1,M]$. Consequently, by time invariance of the dynamics, $\cM$ is admissible invariant with respect to \eqref{eq_sys}.
To show that $\cM$ is maximal, we assume that there exists an admissible invariant set $\cW\subseteq\cX$  satisfying $\cW\nsubseteq\cM$.
 Then, the set ${\cW_{\cL}}:=\{y\in\R^{N}: (\exists x\in\cW: y:=x^{[\cL]})\}$, $\cW_{\cL}\subseteq\cV\cap\cX^{[\cL]}$, is admissible invariant with respect to \eqref{eq_sysd!} and moreover there exists
a vector $y_0\in{\cW}_{\cL}$ such that  $y_0\notin \cS_{\ok}$. 
Taking into account that $\cV$ is invariant under the dynamics \eqref{eq_sysd!},  the last relation implies that for the vector $x_0\in\cW$, where $y_0=x_0^{[\cL]}$, it holds that $y (\ok)\notin\cX^{[\cL]}\cap\cV$, or,  $x(\ok)\notin \cX$, thus, the set $\cW$ is not admissible invariant and we have reached a contradiction. Consequently, $\cW\subseteq \cM$ and $\cM$ is the maximal admissible invariant set.
\end{proof}

Theorem~\ref{theorem1} establishes that the set iteration defined by the pre-image map and initialized with the intersection between the algebraic variety $\cV$ and the lifted set $\cX$ is convergent. Moreover,
the maximal invariant set for the System \eqref{eq_sys} is retrieved directly, by applying the lowering operation on that fixed point.

As discussed and analyzed in the following section, the involved computations at each iteration for the set sequence are linear. However, checking the convergence condition $\cS_{\ok}=\cS_{\ok+1}$ is equivalent to verifying equivalence between two algebraic varieties, a problem which is known to be NP-hard. The following result establishes that the maximal invariant set has an alternative and equivalent characterization. Moreover, the involved convergence criterion in that case involves checking equivalence between two polytopes, which is known to require the solution, at the worst case, of a series of linear programs only. As it is explained below, this alternative approach comes at the cost of possibly introducing redundancies on the description of the maximal invariant set, which however can be removed algorithmically in a post-processing step.

\begin{theorem}\label{theorem2}
Consider the System \eqref{eq_sys}, the constraint set \eqref{eq_set}, the set sequence $\{\cS_i\}_{i\geq 0}$ generated by \eqref{eq_setseq2} with $$\cS_0:=\cX^{[\cL]}$$
and any compact set $\cB\subset\R^{N}$  satisfying $\cV\cap\cX^{[\cL]}\subseteq \cB$. Then, there exists a finite integer $\ok\geq 1$  such that
$$\cS_{\ok}\cap\cB=\cS_{\ok+1}\cap \cB$$ and the maximal admissible invariant set $\cM$ with respect to the System \eqref{eq_sys} and the constraints \eqref{eq_set} is
$\cM=\projection{\cS_{\ok}}$.
\end{theorem}
\begin{proof}
Under Assumption~\ref{assumption1}, there exist scalars $\Gamma\geq 1$, $\varepsilon\in (0,1)$
such that $\| y(t)\|\leq \Gamma \varepsilon^t\| y(0)\|$, for all $y(0)\in\R^{N}$,  $t\geq 0$ and $y(t)$ satisfying \eqref{eq_sysd!}.
Moreover, consider the number
\begin{equation*}
a:=\max\{\lambda: \lambda\cB\subseteq \cB\cap\cX^{[\cL]} \},
\end{equation*}
 and the integer  $\ok=\left\lceil \log_{\varepsilon} \frac{a}{\Gamma} \right\rceil.$
Then, $y(0)\in\cB\cap\cX^{[\cL]}$ implies $y(t)\in\cB\cap\cX^{[\cL]}$, for all $t\geq \ok$.
The rest of the proof follows the same steps of the proof of Theorem~\ref{theorem1}.
\end{proof}

It is worth observing that the sets $\cS_i$, $i\geq 0$ in Theorem~\ref{theorem2} are polyhedral sets.

\begin{remark}
We note that the crucial requirement for this alternative characterization of the maximal admissible invariant set in Theorem~\ref{theorem2} is the boundedness of the set $\cB$,
allowing for the criterion $\cS_{\ok}\cap\cB=\cS_{\ok+1}\cap\cB$ to be verified for a finite integer $\ok\geq 1$.
\end{remark}

The following result applies standard results from the literature to the studied setting, providing a third alternative characterization
of the maximal admissible invariant set, possibly at the cost of adding redundancies in the pre-image map computations.

\begin{theorem}\label{theorem3}
Consider the System \eqref{eq_sys}, the constraint set \eqref{eq_set}, the set sequence $\{\cS_i\}_{i\geq 0}$ generated by \eqref{eq_setseq2} with $$\cS_0:=\cB\cap\cX^{[\cL]},$$
where $\cB\subset\R^{N}$ is a compact polytopic set which contains the origin in its interior and satisfies $\cV\cap\cX^{[\cL]}\subset \cB$. Then, there exists a finite integer $\ok\geq 1$  such that
$$\cS_{\ok}=\cS_{\ok+1}$$ and the maximal admissible invariant set $\cM$ with respect to the System \eqref{eq_sys} and the constraints \eqref{eq_set} is
$\cM=\projection{\cS_{\ok}}$.
\end{theorem}
\begin{proof}
From Fact~\ref{fact2}, the set $\cX\cap\cV$ is compact, thus, by construction and Assumption~\ref{assumption2}, the set $\cS_0$ is compact and contains the origin in its interior.
Consequently, under Assumption~\ref{assumption1}, from \cite[Ch. 5]{Blanchini08} there exists a finite integer $\ok$ such that $\cS_{\ok}$ is the maximal admissible invariant set with respect to $\cX^{[\cL]}$. Taking into account Proposition~\ref{lemma1} and observing that $\cV\cap\cX^{[\cL]}\subset \cB$ and that $\cV$ is invariant under \eqref{eq_sysd!}, the result follows.
\end{proof}

\begin{remark}
We can replace boundedness of $\cB$ in Theorem~\ref{theorem3} with requiring $\cB$ to be a symmetric polyhedron whose defining matrix in the half-space description satisfies an observability condition with at least a member of the set\footnote{ $\conv{\cdot}$ stands for the convex hull.} $\conv{  \{A_1^{[\cL]},...,A_M^{[\cL]}  \}   }$. For more details see, e.g., \emph{\cite{gilbert:tan:1991}}.
\end{remark}

\section{Implementation }\label{section4}

In this section, we present three algorithmic procedures for computing the maximal admissible invariant set for the System \eqref{eq_sys} subject to the constraints \eqref{eq_constraints}.
In detail, we present an efficient way to realize the set sequences and verify the convergence criteria of the theoretical results of the previous section.
First, we establish the relationship between the set sequences generated in Theorem~\ref{theorem1} and Theorem~\ref{theorem2}.

\begin{table}\label{table11}
\centering
\begin{tabular}{|c|c|c|} \hline

\begin{tabular}[x]{@{}c@{}}\textbf{Theorem /} \\ \textbf{Algorithm} \end{tabular} & \begin{tabular}[x]{@{}c@{}} \textbf{Initial} \textbf{ set} $\cS_0$\end{tabular} &\begin{tabular}[x]{@{}c@{}}\textbf{Convergence criterion} \end{tabular} \\ \hline

$1$ & $\cV\cap\cX^{[\cL]}$ & $\cS_{\ok+1}=\cS_{\ok}$ \\ \hline

$2$ & $\cX^{[\cL]}$ & $\cS_{\ok+1}\cap\cB=\cS_{\ok}\cap\cB$ \\ \hline

$3$ & $\cB\cap\cX^{[\cL]}$  & $\cS_{\ok+1}=\cS_{\ok}$  \\ \hline

\end{tabular}
\caption{Summary of the results of section~\ref{section3} and section~\ref{section4}: Each set sequence obeys the update relation $\cS_{i+1}=\cC(\cS_i)\cap\cS_0$.
The sets $\cV\subset\R^{N}$, $\cX^{[\cL]}$ are defined in \eqref{eq_manifold} and \eqref{eq_set_lifted} respectively,
the compact set $\cB\subset\R^N$ satisfies $\cB\supseteq\cV\cap\cX^{[\cL]}$ in Theorem/Algorithm $2$,
while the compact polytopic set $\cB\subset\R^N$ in  Theorem/Algorithm $3$ contains the origin in its interior and satisfies $\cB\supseteq \cV\cap\cX^{[\cL]}$.}
\end{table}

\begin{fact}\label{fact3}
Consider any two sets $\cY\subset\R^N,\cZ\subset\R^N$, the pre-image map~\eqref{eq_preimage} and the Veronese variety $\cV$ \eqref{eq_manifold}. Then, (i) $\cC(\cY\cap\cZ)=\cC(\cY)\cap\cC(\cZ)$ ,and (ii)
$\cV \subseteq \cC(\cV)$.
\end{fact}
\begin{proof}
Statement (i) follows from the definition \eqref{eq_preimage}, while (ii) follows from the fact that $\cV$ is invariant with respect to the System \eqref{eq_sysd!}.
\end{proof}

\begin{algorithm}\label{algorithm1}
\caption{\textbf{Inputs}:  $\cA:=\{A_1,...,A_M\}\subset\R^{n\times n}$, $\cX=\{x\in\R^n: c_i(x)\leq 1, i\in[1,p]\}$
 \textbf{Output:} The maximal admissible invariant set $\cM$. }
\begin{algorithmic}[1]
\STATE  Extract $\cL\subseteq \{1,d\}$, $g_{i}\in\R^{N}$, satisfying $g_i^\top x^{[\cL]}=c_i(x)$, $i\in[1,p]$.
\STATE $i\leftarrow 0$,  eq$\leftarrow 0$, $\cZ_0\leftarrow \cX^{[\cL]}$, $\cY_0\leftarrow \cZ_0\cap\cV$
\WHILE{eq$=0$}
\STATE $\cZ_{i+1}\leftarrow \cC(\cZ_i)$, as in \eqref{eq_polpre1}, \eqref{eq_polpre2}
\STATE Compute the minimal description of $\cZ_{i+1}$ (Appendix~\ref{appendix1})
\STATE $\cY_{i+1}\leftarrow\cZ_{i+1}\cap\cV$
\STATE Compute the minimal description of $\cY_{i+1}$ (Appendix~\ref{appendix2})
\IF{$\cY_{i+1}=\cY_{i}$}
\STATE eq$\leftarrow 1$
\ENDIF
\STATE $i\leftarrow i+1$
\ENDWHILE
\STATE $\cM\leftarrow\projection{\cZ_{i}}$
\end{algorithmic}
\end{algorithm}

\begin{lemma}\label{lemma3}
Let $\{\cY_i\}_{i\geq 0}$, $\{\cZ_i  \}_{i\geq 0}$ be the set sequences generated by \eqref{eq_setseq2} with initial conditions
$\cY_0=\cV\cap\cX^{[\cL]}$ and $\cZ_0=\cX^{[\cL]}$ respectively.
Then, the relation
\begin{equation}\label{eq_lem3_1}
\cY_i=\cZ_i\cap\cV, \quad \forall i\geq 0
\end{equation}
 holds.
\end{lemma}
\begin{proof}
For $i=0$, \eqref{eq_lem3_1} holds by definition. Suppose that \eqref{eq_lem3_1} holds for $i=k$.
Then, for $i=k+1$ and taking into account Fact~\ref{fact3}, it follows that
\begin{align*}
\cY_{k+1}&=\cC(\cZ_k\cap\cV)\cap\cV\cap\cX^{[\cL]}=\cC(\cZ_k)\cap\cC(\cV)\cap\cV\cap\cZ_0 \\
& =\cC(\cZ_k)\cap\cZ_0\cap \cV=\cZ_{k+1}\cap\cV,
\end{align*}
thus, relation \eqref{eq_lem3_1} holds for all $i\geq 0$.
\end{proof}

Lemma~\ref{lemma3} states that the set sequence defined in Theorem~\ref{theorem1} can be generated in two steps and in specific by computing first the pre-image map of a polyhedral set and consequently its intersection with the manifold $\cV$.

\begin{remark}\label{remark5}
In Line~$4$ of Algorithm~1 the computation of the pre-image map of a polyhedral set is required.
To this end, let $\cZ_i\subset\R^N$ be the polyhedral set computed at iteration $i$ in half-space representation, i.e.,
\begin{equation}\label{eq_polyhedral}
\cZ_i:=\{ y\in\R^N: G_iy\leq 1_{p_i}  \},
\end{equation}
where $G_i\in\R^p_i\times N$ and $p_i\geq 1$. Then, the pre-image map $\cC(\cS)$ with respect to the System \eqref{eq_sysd!} is
\begin{align}
\cC(\cZ_i) & =\{y\in\R^N: G_iA^{[\cL]}_jy\leq 1_{p_i}, j=1,...,M \} \nonumber \\
& =  \{y\in\R^N: G^\star y\leq 1_{p^\star}\}, \label{eq_polpre1}
\end{align}
 where $p^\star=pM$ and
\begin{equation}\label{eq_polpre2}
G^\star=\left[ \quad (G_iA_1^{[\cL]})^{\top}\quad \ldots \quad (G_iA_M^{[\cL]})^\top\right]^\top.
\end{equation}
The number of hyperplanes that describe the set $\cZ_i$ is bounded by $p^M$, where $p$ is the number of hyperplanes that describe the set $\cX^{[\cL]}$
and $M$ is the number of matrices defining the system \eqref{eq_sys}. However, in practice the number of hyperplanes, or equivalently, the size  of the matrices $G_i$, $i\geq 0$
that are required to describe $\cZ_i$ is significantly smaller.
\end{remark}
In Appendix \ref{appendix1}, a procedure of
computing the minimal representation of the set $\cZ_i$, required in Line~$5$ of Algorithm~$1$ is described.

\begin{algorithm} \label{algorithm2}
\caption{\textbf{Inputs}:  $\cA:=\{A_1,...,A_M\}\subset\R^{n\times n}$, $\cX=\{x\in\R^n: c_i(x)\leq 1, i\in[1,p]\}$, compact polytopic set $\cB\supset \cX^{[\cL]} \cap\cV$ (Appendix \ref{appendix3})
 \textbf{Output:} The maximal admissible invariant set $\cM$. }
\begin{algorithmic}[1]
\STATE  Extract $\cL\subseteq \{1,d\}$, $g_{i}\in\R^{N}$, satisfying $g_i^\top x^{[\cL]}=c_i(x)$, $i\in[1,p]$.
\STATE $i\leftarrow 0$, eq$\leftarrow 0$, $\cZ_0\leftarrow \cX^{[\cL]}$
\WHILE{eq$=0$}
\STATE $\cZ_{i+1}\leftarrow \cC(\cZ_i)$, as in \eqref{eq_polpre1}, \eqref{eq_polpre2}
\STATE Compute the minimal description of $\cZ_{i+1}$ (Appendix~\ref{appendix1})
\IF{$\cY_{i+1}\cap\cB=\cY_{i}\cap\cB$}
\STATE eq$\leftarrow 1$
\ENDIF
\STATE $i\leftarrow i+1$
\ENDWHILE
\STATE $\cM\leftarrow\projection{\cY_{i}}$
\STATE (\emph{optional}) Compute the  minimal representation of $\cM$ (Appendix \ref{appendix2})
\end{algorithmic}
\end{algorithm}

The set $\cY_{i}=\cZ_{i}\cap\cV$ in Line $6$ of Algorithm~$1$ has a straightforward description. In specific, if $\cZ_i$ is described by \eqref{eq_polyhedral},
it holds that
\begin{equation}
\cY_i=\{y\in\R^N:(\exists x\in\R^n: y=x^{[\cL]},  G_i y\leq 1_{p_i} ) \}.
\end{equation}
However, computing the minimal description of the set $\cY_{i+1}$ in Algorithm~$1$, or in other words removing the redundant polynomial inequalities of the set $\projection{\cY_{i}}$, is equivalent to
verifying equivalence between two algebraic varieties.
The approach taken in this paper is to iteratively check for redundancy of each hyperplane of the set $\cY_i$, or equivalently,
to check for redundant polynomial inequalities of the set $\projection{\cY_i}$.
In Appendix~\ref{appendix2}, a possible approach for tackling this problem, based on a version of the Positivstellensatz  \cite{Putinar:93}, \cite[Theorem 3.138]{Parrilo:12},  is presented.

\begin{algorithm} \label{algorithm3}
\caption{\textbf{Inputs}:  $\cA:=\{A_1,...,A_M\}\subset\R^{n\times n}$, $\cX=\{x\in\R^n: c_i(x)\leq 1, i\in[1,p]\}$, compact polytopic set $\cB\supset \cX^{[\cL]} \cap\cV$ containing the origin in its interior (Appendix \ref{appendix3})
 \textbf{Output:} The maximal admissible invariant set $\cM$. }
\begin{algorithmic}[1]
\STATE  Extract $\cL\subseteq \{1,d\}$, $g_{i}\in\R^{N}$, satisfying $g_i^\top x^{[\cL]}=c_i(x)$, $i\in[1,p]$.
\STATE $i\leftarrow 0$,  eq$\leftarrow 0$, $\cZ_0\leftarrow \cB\cap \cX^{[\cL]}$
\WHILE{eq$=0$}
\STATE $\cZ_{i+1}\leftarrow \cC(\cZ_i)$, as in \eqref{eq_polpre1}, \eqref{eq_polpre2}
\STATE Compute the minimal description of $\cZ_{i+1}$ (Appendix~\ref{appendix1})
\IF{$\cY_{i+1}\cap\cB=\cY_{i}\cap\cB$}
\STATE eq$\leftarrow 1$
\ENDIF
\STATE $i\leftarrow i+1$
\ENDWHILE
\STATE $\cM\leftarrow\projection{\cY_{i}}$
\STATE (\emph{optional}) Compute the  minimal representation of $\cM$ (Appendix \ref{appendix2})
\end{algorithmic}
\end{algorithm}

Contrary to Algorithm~$1$, Algorithms~$2$ and $3$ are based solely on linear operations and on solving linear programs.
It is worth observing that the number of iterations needed in Algorithms~$2$ and $3$ to recover the maximal admissible invariant set
is lower bounded by the number of iterations needed in Algorithm~$1$. This is the cost that has to be paid in order to avoid computing the minimal representation of the set $\cZ_{i+1}\cap\cV$ at each iteration in Algorithm~$1$.
Naturally, if one is interested in the minimal representation of the maximal admissible invariant set $\cM$,
the approach described in Appendix \ref{appendix2} can be used in a single post-processing step in Line $12$ of Algorithms~$2$~and~$3$.

\begin{rexample}
We implement Algorithm~$2$ in order to compute the maximal admissible invariant set. To this end, we first choose a compact polytopic set
\begin{equation*}
\cB=\{y\in\R^3: -\sqrt{2}\leq y_2\leq \sqrt{2}, 0\leq y_i\leq 1, i=1,3   \}
\end{equation*}
such that $\cB\supset\cV\cap\cX^{[\cL]}$.
As described above, we set $\cZ_0=\cX^{[\cL]}$. The Algorithm $2$ converges after $8$ iterations, i.e., the relation $\cZ_{8}\cap\cB=\cZ_7\cap\cB$
is satisfied. In Figure \ref{Ex_comp4} the set $\cZ_7$ is shown in blue while the hyperplanes that define the set $\cX^{[\cL]}$ are also shown in grey.
In Figure~\ref{Ex_comp3}, the maximal invariant set $\projection{\cZ_{7}}$ together with the constraint set $\cX$ are shown. It is worth observing that the maximal invariant set is not convex, as expected.
The level curves of the polynomial functions that define the maximal invariant set are also shown. In specific, there are $14$ polynomials in total which define the set, out of which $5$ of them are redundant and have been identified by applying the post-processing step (Line $12$ of Algorithm $2$).


\begin{figure}[h]
\begin{center}
\includegraphics[height=4cm]{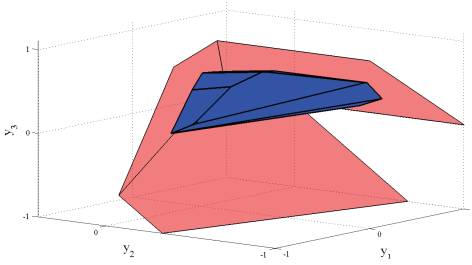}
\end{center}
\caption{Running example, the set $\cZ_7$ (blue) and the hyperplanes that define the set $\cX^{[2]}$ (red).}
\label{Ex_comp4}
\end{figure}
\end{rexample}

\begin{figure}[h]
\begin{center}
\includegraphics[height=4cm]{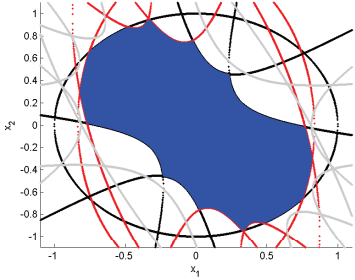}
\end{center}
\caption{Running example, the maximal admissible invariant set $\projection{\cZ_{7}}$ (blue) and the level curves of the polynomial inequalities which define $\projection{\cZ_7}$:
The polynomial constraints inherited from $\cX$ are in black, the added active constraints are in red, whereas the added redundant constraints are shown in grey.}
\label{Ex_comp3}
\end{figure}

%
%
%
%
%
%
%

Finally, two properties of the maximal invariant set $\cM$ which are inherited from the constraint set  $\cX$ are summarized below.
\begin{proposition}
Consider the System \eqref{eq_sys} subject to  constraints \eqref{eq_set} and let $\cM$ be the maximal admissible invariant set.
Then, the following hold:

\noindent \quad \quad (i) $\cM$ is the sub-level set of a max-polynomial function of at most degree $d$, described by a finite number of pieces.

\noindent \quad \quad  (ii) If $\cX$ is convex, then $\cM$ is convex.

\end{proposition}

\begin{proof}
(i) Follows directly from Algorithms~$2$, $3$ and in specific from the facts that the sets $\cZ_i$, $i\geq 0$, are polyhedral and that the algorithm terminates in finite time.

\noindent (ii)  Taking into account Theorem~\ref{theorem1}, it is enough to show that the pre-image map $\cC(\cS)$ with respect to \eqref{eq_sys} is always convex when  $\cS:=\{x\in\R^n: c_i(x)\leq 1, i=1,...,p   \}\subset\R^n$ is convex.
Since $\cC(\cS)=\{ x\in\R^n : c_i(A_j(x))\leq 1, i=1,...,p, j=1,...,M  \}$ and taking into account \cite[Ch. 3]{boyd:vandenberghe:2004} that
the composition of a convex function and a linear function is convex and the maximum of convex functions is convex, it follows that $\cC(\cS)$ is convex, thus, the maximal invariant set is convex.
\end{proof}

\section{Numerical examples}\label{section5}

\begin{example}
We consider a linear time invariant system
\begin{equation}
x(t+1)=Ax(t),
\end{equation}
with $A=\begin{bmatrix}  1.0216 & 0.3234 \\ -0.6597  &  0.5226\end{bmatrix}$. We are
interested in computing the maximal admissible invariant set when the constraint set $\cX\subset\R^2$ is the unit circle.
For all three Algorithms $1$-$3$, the maximal admissible invariant set $\cM$ is recovered in exactly $6$ iterations. For comparison, we compute the maximum invariant ellipsoid $\cE_{\max}$ contained in $\cX$ by solving a linear matrix inequality problem, (for details see, e.g., \cite[Ch. 5]{boyd:ghaoui:feron:balakrishnan:1994}). As expected, we can see in Figure~\ref{numex1} that $\cE_{\max}\subset\cM$.

\begin{figure}[h]
\begin{center}
\includegraphics[height=4.5cm]{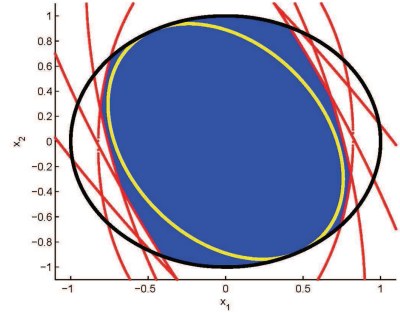}
\end{center}
\caption{Example 1, the maximal admissible invariant set $\cM$ is in blue, the level curves of the polynomial inequalities which define the set $\cM$ are in red, while the maximum invariant inscribed ellipsoid  $\cE_{\max}$ is shown in yellow.}
\label{numex1}
\end{figure}
\end{example}


\begin{example} \label{Example6}
We consider the System \eqref{eq_sys} with $\cA=\{A_1,A_2\}$, where
$A_1=\begin{bmatrix}  0.2137 & 1.2052 \\ -0.2125  &  0.1703\end{bmatrix}$,
$A_2=\begin{bmatrix}  -0.3576 & 1.0351 \\ 0.3290  &  0.3514\end{bmatrix}$.
The constraint set $\cX$ \eqref{eq_set} is non-simply connected and is described by the intersection of the unit circle and the
complements of two circles and an ellipse. In this setting we have $\cL=\{1,2\}$.  By applying Algorithm~$2$, the maximal admissible invariant set $\cM$ is retrieved in $5$ iterations and is described by $36$ polynomial inequalities. It is worth observing that the set $\cM$ is not connected.

%
%
%
%
\begin{figure}[h]
\begin{center}
\includegraphics[height=4.5cm]{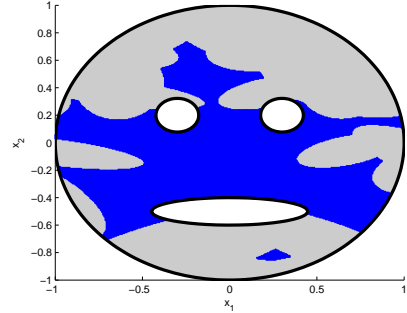}
\end{center}
\caption{Example 2, the maximal admissible invariant set $\cM$ is shown in blue, while the constraint set $\cX$ is shown in grey.}
\label{numex2}
\end{figure}
\end{example}


\section{Conclusions and future work}
In this work, we studied the computation of the maximal admissible invariant
set for switching linear discrete time systems that are subject
to semi-algebraic constraints. In this setting, the maximal admissible invariant set might be non-convex, or even non-connected.
However, we showed that, despite the complexity of these constraints, the computation of the maximal admissible invariant set
can be reduced to a problem with much simpler linear constraints (i.e., a polytopic constraint set).
The approach consists in applying the Veronese embedding and consequently
lifting the system and the constraint set in a higher dimensional
space, allowing for efficient reachability operations.

This comes at the price of inflating the dimension, and hence, the number of variables, calling for a careful study of the computational burden necessary for computing these invariant sets.  In this work, we made a first step in that direction by presenting three different algorithms, with different advantages. Moreover, we suggested several subroutines that are required. We leave for further research the question of precisely comparing the efficiency between the established algorithms and choosing the optimal mathematical tools, e.g., for the removal of redundant constraints.    In addition, we plan to investigate how the approach can be applied to systems with inputs, and how it can be utilised for systems where the maximal admissible invariant set is a polytope, but one would like to approximate it with much fewer constraints.

%

\begin{thebibliography}{10}

\bibitem{AhmadiJungers:13}
A.~A. Ahmadi and R.~M. Jungers.
\newblock { Switched stability of nonlinear systems via SOS-convex Lyapunov
  functions and semidefinite programming }.
\newblock In {\em American Control Conference}, pages 2686--2700, Boston, MA,
  USA, 2005.

\bibitem{ABL:14}
N.~Athanasopoulos, M.~Lazar, and G.~Bitsoris.
\newblock {Property-preserving convergent sequences of invariant sets for
  linear discrete-time systems}.
\newblock In {\em 21st International Symposium on Mathematical Theory of
  Networks and Systems}, pages 1280--1286, Groningen, The Netherlands, 2014.

\bibitem{Aubin:11}
J.~P. Aubin, A.~M. Bayen, and P.~Saint-Pierre.
\newblock {\em Viability Theory: New Directions}.
\newblock Springer, {Heidelber Dordrecht London New York }, 2011.

\bibitem{Barber:96}
C.~B. Barber, D.~P. Dobkin, and H.~Huhdanpaa.
\newblock {The Quickhull Algorithm for Convex Hulls}.
\newblock {\em ACM Transactions on Mathematical Software}, 22:469--483, 1996.

\bibitem{Belta:07}
C.~Belta, V.~Isler, and G.~J. Pappas.
\newblock Discrete abstractions for robot motion planning and control in
  polygonal environments.
\newblock {\em IEEE Transactions on Robotics}, 21(5):864--874, 2005.

\bibitem{bertsekas:1972}
D.~P. Bertsekas.
\newblock {Infinite--Time Reachability of State--Space Regions by Using
  Feedback Control}.
\newblock {\em {IEEE Transactions on Automatic Control}}, 17(5):604--613, 1972.

\bibitem{Bitsoris:88}
G.~Bitsoris.
\newblock {On the positive invariance of polyhedral sets for discrete-time
  systems}.
\newblock {\em Systems and Control Letters}, 11:243--248, 1988.

\bibitem{BitsSorin:13}
G.~Bitsoris and S.~Olaru.
\newblock {Further Results on the Linear Constrained Regulation Problem}.
\newblock In {\em 21st IEEE Mediterranean Conference on Control and
  Automation}, pages 824--830, Platanias, Greece, 2013.

\bibitem{blanchini:1999}
F.~Blanchini.
\newblock {Set Invariance in Control -- A Survey}.
\newblock {\em {Automatica}}, 35(11):1747--1767, 1999.
\newblock {Survey Paper}.

\bibitem{Blanchini08}
F.~Blanchini and S.~Miani.
\newblock {\em Set-theoretic methods in control}.
\newblock Systems \& Control: Foundations \& Applications. Birkh\"{a}user,
  Boston, MA, 2008.

\bibitem{Parrilo:12}
G.~Blekherman, P.~A. Parrilo, and R.~R. Thomas.
\newblock {\em {Semidefinite Optimization and Convex Algebraic Geometry}},
  volume~13 of {\em {MOS-SIAM Series on Optimization}}.
\newblock SIAM, 2012.

\bibitem{BloNes:05}
V.~D. Blondel and Y.~Nesterov.
\newblock Computationally efficient approximations of the joint spectral
  radius.
\newblock {\em SIAM Journal of Matrix Analysis}, 27:256--272, 2005.

\bibitem{boyd:ghaoui:feron:balakrishnan:1994}
S.~Boyd, L.~E. Ghaoui, E.~Feron, and V.~Balakrishnan.
\newblock {\em Linear Matrix Inequalities in System and Control Theory}.
\newblock Studies in Applied Mathematics. SIAM, 1994.

\bibitem{fukuda:2000}
K.~Fukuda.
\newblock Frequently asked questions in polyhedral computation.
\newblock Official website:
  http://www.ifor.math.ethz.ch/$\sim$fukuda/polyfaq/polyfaq.html.

\bibitem{gilbert:tan:1991}
E.~G. Gilbert and K.~T. Tan.
\newblock Linear systems with state and control constraints: the theory and
  application of maximal output admissible sets.
\newblock {\em IEEE Transactions on Automatic Control}, 36(9):1008--1020, 1991.

\bibitem{Gilbert91}
E.~M. Gilbert and K.~T. Tan.
\newblock Linear systems with state and control constraints: {T}he theory and
  application of maximal output admissible sets.
\newblock {\em IEEE Transactions on Automatic Control}, 36(9):1008--1020, 1991.

\bibitem{gutman:87}
P.~O. Gutman and M.~Cwikel.
\newblock An algorithm to find maximal state constraint sets for discrete-time
  linear dynamical systems with bounded control and states.
\newblock {\em {IEEE Transactions on Automatic Control}}, 32:251--254, 1987.

\bibitem{Hennet:95}
J.~C. Hennet.
\newblock {Discrete Time Constrained Linear Systems}.
\newblock {\em Control and Dynamic Systems, Leondes Ed. Academic Press},
  71:157--213, 1995.

\bibitem{HuLin:2003}
T.~Hu and Z.~Lin.
\newblock {Composite quadratic Lyapunov functions for constrained control
  systems}.
\newblock {\em IEEE Transactions on Automatic Control}, 48(3):440--450, 2003.

\bibitem{RantzerANA98}
M.~Johansson and A.~Rantzer.
\newblock Computation of piecewise quadratic {Lyapunov} functions for hybrid
  systems.
\newblock {\em IEEE Transactions on Automatic Control}, 43(4):555--559, 1998.

\bibitem{Jungers:09}
R.~M. Jungers.
\newblock {\em {The joint spectral radius: theory and applications}}, volume
  385 of {\em Lecture Notes in Control and Information Sciences}.
\newblock Springer, 2008.

\bibitem{Khalil02}
H.~Khalil.
\newblock {\em Nonlinear Systems, Third Edition}.
\newblock Prentice Hall, 2002.

\bibitem{MDA:13}
M.~Lazar, A.~I. Doban, and N.~Athanasopoulos.
\newblock {On stability analysis of discrete--time homogeneous dynamics}.
\newblock In {\em 17th International Conference on System Theory, Control and
  Computing}, pages 1--8, Sinaia, Romania, 2013.

\bibitem{LinAnts:09}
H.~Lin and P.~J. Antsaklis.
\newblock Stability and stabilizability of switched linear systems : a survey
  of recent results.
\newblock {\em IEEE Transactions on Automatic Control}, 54:308--322, 2009.

\bibitem{MOlchPyat:89}
A.~P. Molchanov and Y.~S. Pyatnitsky.
\newblock Criteria of asymptotic stability of differential and difference
  inclusions encountered in control theory.
\newblock {\em Systems and Control Letters}, 13:59--64, 1989.

\bibitem{Papa:05}
A.~Papachristodoulou and S.~Prajna.
\newblock { A tutorial on sum of squares techniques for systems analysis}.
\newblock In {\em American Control Conference}, pages 2686--2700, Boston, MA,
  USA, 2005.

\bibitem{Parrilo:00}
P.~Parrilo.
\newblock {\em {Structured Semidefinite Programs and Semialgebraic Geometry
  Methods in Robustness and Optimization}}.
\newblock PhD thesis, {California Institute of Technology, CA, USA}, 2000.

\bibitem{parrilo:08}
P.~A. Parrilo and A.~Jadbabaie.
\newblock {Approximation of the joint spectral radius using sum of squares}.
\newblock {\em Linear Algebra and Its Applications}, 428(10):2385--2402, 2008.

\bibitem{Prajna03}
S.~Prajna and A.~Papachristodoulou.
\newblock Analysis of switched and hybrid systems - beyond piecewise qudratic
  methods.
\newblock In {\em 22nd American Control Conference}, pages 2779--2784, Denver,
  Colorado, 2003.

\bibitem{PPP:02}
S.~Prajna, A.~Papachristodoulou, and P.~A. Parillo.
\newblock { Introducing SOSTOOLS: A general purpose sum of squares programming
  solver}.
\newblock In {\em 41st IEEE Conference on Decision and Control}, pages
  741--746, Las Vegas, USA, 2002.

\bibitem{Putinar:93}
M.~Putinar.
\newblock {Positive polynomials on compact semi-algebraic sets}.
\newblock {\em Indiana University Mathematics Journal}, 42:969--984, 1993.

\bibitem{RotaStrang:60}
G.~C. Rota and W.~G. Strang.
\newblock A note on the joint spectral radius.
\newblock {\em Proceedings of the Netherlands Academy}, 22:379--381, 1960.

\bibitem{boyd:vandenberghe:2004}
{S.~Boyd and L.~Vandenberghe}.
\newblock {\em {Convex Optimization}}.
\newblock {Cambridge University Press}, {Cambridge, England}, 2004.

\bibitem{Seidenberg:54}
A.~Seidenberg.
\newblock {A New Decision Method for Elementary Algebra}.
\newblock {\em Annals of Mathematics}, 60:365--374, 1954.

\bibitem{Tarski:48}
A.~Tarski.
\newblock {\em {A decision method for elementary algebra and geometry}}.
\newblock Rand Corporation Publication, 1948.

\bibitem{Raphael:14}
G.~Vankeerberghen, J.~Hendrickx, and R.~M. Jungers.
\newblock {JSR: A toolbox to compute the joint spectral radius}.
\newblock In {\em 17th International Conference on Hybrid systems: Computation
  and Control}, pages 151--156, Berlin, Germany, 2014.

\bibitem{zelentsovsky:1994}
A.~Zelentsovsky.
\newblock {Nonquadratic Lyapunov Functions for Robust Stability Analysis of
  Linear Uncertain systems}.
\newblock {\em {IEEE Transactions on Automatic Control}}, 39(1):135--138, 1994.

\bibitem{Ziegler07}
G.~M. Ziegler.
\newblock {\em Lectures on Polytopes, Updated Seventh Printing}.
\newblock Springer, 2007.

\end{thebibliography}
%
%

\appendix

We describe computationally efficient procedures that can be used to realize intermediate steps in Algorithms~$1$-$3$.

\section{\small{ \textbf{Minimal description of polyhedral sets}}} \label{appendix1}

Finding the minimal representation of a polyhedral set $\cS\subset\R^N$ is a  well-studied problem, see e.g. \cite{fukuda:2000}, \cite{Ziegler07}, and it is generally accepted that it can be solved efficiently for relatively low dimensions $N$.  It is worth noting that there are methods in which a set of redundant inequalities is removed at each step rather than a single inequality, see e.g., convex hull algorithms \cite{Barber:96} which are directly applicable by the duality of the problems.

 In what follows, we present a simple way to remove a redundant hyperplane in the description of $\cS$ by solving a linear program.
To this end, consider the set $\cS=\{y\in\R^N: g_i^\top y\leq 1, i=1,...,p  \}$, $p\geq 2$. Then, $\cS=\{ y\in\R^N: g_i^\top y\leq 1, i=1,...,p, i\neq j \}$ for some $j\in [1,p]$ if and only if the optimal cost of the linear program
\begin{equation*}
\max_{x} g_j^\top x
\end{equation*}
subject to
\begin{align*}
g_i^\top x\leq 1, \quad \forall i\in[1,p]\setminus \{j\},
\end{align*}
satisfies $g_j^\top x^\star<1$.

\section{\small{ \textbf{Minimal description of semi-algebraic sets}}} \label{appendix2}
Finding the minimal representation of a semi-algebraic set is a much more difficult problem when the polynomials defining the set are not linear.
Deciding for redundancy of a polynomial inequality in the description of a semi-algebraic set can be performed using the Tarski-Seidenberg elimination theorem \cite{Tarski:48}, \cite{Seidenberg:54}.
This implies that the redundancy removal problem is decidable. However, despite its generality, the drawback of the corresponding algorithmic method is its computational complexity, which increases at least exponentially with the number of unknowns.

In what follows we propose a way to remove a redundant polynomial inequality by transforming the problem in a series of semidefinite programs.
It is worth stating that this approach poses sufficient conditions for checking redundancy, however in a computationally efficient manner, see e.g., \cite{Parrilo:00}.
To this end, let $\cS\subset\R^n$,
\begin{equation} \label{eq_app1}
\cS=\projection{\cY_i\cap\cV}=\{x\in\R^n: g_i^\top x^{[\cL]}\leq 1, i=1,...,p \}.
\end{equation}
The next result is an application of Putinar's theorem \cite{Putinar:93}, \cite[Theorem 3.138]{Parrilo:12}.
\begin{proposition} \label{proposition5}
Consider the set $\cS\subset\R^n$ \eqref{eq_app1}. Then, there exists an integer $j\in[1,p]$ such that
\begin{equation} \label{eq_app2}
\cS=\{x\in\R^n: g_i^{\top}x^{[\cL]}\leq 1, i=1,...,p, i\neq j    \}
\end{equation}
if and only if there exist polynomials $s_0(x)$, $s_i(x)$, $i\in [1,p]\setminus \{j\}$, such that
\begin{equation}
1-g_j^\top x^{[\cL]}=s_0^2(x) + \sum_{i=1,i\neq j}^{p} s_i^2(x)(1-c_i(x)).
\end{equation}
\end{proposition}

Proposition~\ref{proposition5} provides a necessary and sufficient condition of identifying redundant inequalities $g_j^\top x^{[\cL]}\leq 1$ in the description of the set $\cS$.
However, it is not algorithmically implementable, since the degree of the functions $s_0(x)$, $s_i(x)$, $i\in[1,p]\setminus\{j\}$ can be arbitrarily high.
Nevertheless, by fixing the maximum degree of the polynomials $s_0(x), s_i(x)$, we can formulate the following optimization problem
\begin{equation} \label{eq_app3}
\min_{s_0(x),s_i(x)}\varepsilon
\end{equation}
subject to
\begin{align}
\varepsilon-g_j^\top x^{[\cL]}=s_0^2(x)+ \sum_{i=1,i\neq j}^{p} s_i^2(x)(1-c_i(x)). \label{eq_app4}
\end{align}
The optimization problem \eqref{eq_app3}, \eqref{eq_app4} is equivalent to a semidefinite program, see e.g. \cite{Parrilo:00,PPP:02}.
If the optimal cost is $\varepsilon^\star<1$ for an index $j$, then the set $\cS$ can be described by \eqref{eq_app2}.

\section{\small{ \textbf{Computation of the sets $\cB\subset\R^N$ required for the initialization of Algorithms~$2$ and $3$. }}} \label{appendix3}
Under Assumption~\ref{assumption2}, we can always find polytopic sets $\cB\subset\R^N$ satisfying the  properties in Theorems~\ref{theorem2} and \ref{theorem3}.
In this section we propose one such possible construction.
To this end, we first compute a set $\cB_1\subset\R^n$ such that
$\cB_1 :=\{ x\in\R^n: x_{\min}\leq x\leq x_{\max}  \}$.
Next, we define $\cB_{l_j}$, $l_j\in\cL$,
\begin{equation*}
\cB_{l_j}:=\{ y\in\R^{N_j}: y_{\min}^{l_j}\leq y\leq y_{\max}^{l_j}  \},
\end{equation*}
where
\begin{align*}
y_{\max,k}^{l_j} & := \max\left\{x^\alpha_k \sqrt{\alpha_k!}:  x_i\in \{x_{\min,i},x_{\max,i}\}, i\in[1,n]  \right\}, \\
y_{\min,k}^{l_j} & := \min\left\{x^\alpha_k \sqrt{\alpha_k!}:  x_i\in \{x_{\min,i},x_{\max,i}\}, i\in[1,n]  \right\},
\end{align*}
$k\in [1,N_j]$, $N_j={n+l_j-1 \choose l_j}$
while  each element $y_k$, $k\in[1,N_j]$, corresponds to the monomial $x^{\alpha}$ of the $l_j$-lift of $x$.
The set
\begin{equation*}
\cB:=\cB_{l_1}\times \cB_{l_2}\times\ldots\times\cB_{l_K}
\end{equation*}
is a polytope, can be used to initialize Algorithm~$2$ since $\cB\supset \cX^{[\cL]}\cap\cV$ and
is described by
\begin{equation} \label{eq_app5}
\cB=\{y\in\R^N: R_{\min}\leq y\leq R_{\max}  \},
\end{equation}
where
\begin{align*}
R_{\min} & =\left[y_{\min}^{l_1 \top}\quad \ldots \quad y_{\min}^{l_K \top}	 \right]^{\top}, \\
R_{\max} &=\left[y_{\max}^{l_1 \top}\quad \ldots \quad y_{\max}^{l_K \top}	 \right]^{\top}.
\end{align*}
To recover a set $\cB\subset\R^N$ which can be used for initialization in Algorithm~$3$, it is sufficient to replace $R_{\min}$ in \eqref{eq_app5} with
\begin{equation*}
\hat{R}_{\min,i}=\min\{ -\delta, R_{\min,i}   \},
\end{equation*}
for all $i=1,...,\sum_{l_j\in\cL}{n+l_j-1 \choose l_j}$ and some positive scalar $\delta>0$.
\end{document}